\documentclass[11pt]{amsart}
\usepackage[pdftex]{hyperref}
\usepackage{amsfonts, amsthm, amsmath, amssymb}
\usepackage[all]{xy}
\usepackage[top=1in,bottom=1in,left=1in,right=1in]{geometry}

\newcommand{\bB}{\mathbb{B}}
\newcommand{\bC}{\mathbb{C}}

\newcommand{\bF}{\mathbb{F}}

\newcommand{\bM}{\mathbb{M}}

\newcommand{\bQ}{\mathbb{Q}}
\newcommand{\bR}{\mathbb{R}}

\newcommand{\bZ}{\mathbb{Z}}
\newcommand{\unit}{\mathbf{1}}

\newcommand{\cC}{\mathcal{C}}

\newcommand{\cW}{\mathcal{W}}

\newcommand{\fgl}{\mathfrak{gl}}
\newcommand{\fh}{\mathfrak{h}}
\newcommand{\fH}{\mathfrak{H}}

\newcommand{\fsl}{\mathfrak{sl}}

\DeclareMathOperator{\Aut}{Aut}

\DeclareMathOperator{\End}{End}

\DeclareMathOperator{\Hom}{Hom}
\DeclareMathOperator{\id}{id}

\DeclareMathOperator{\Res}{Res}
\DeclareMathOperator{\Spin}{Spin}

\DeclareMathOperator{\Sym}{Sym}
\DeclareMathOperator{\Tr}{Tr}

\newcommand{\Ind}{\operatorname{Ind}}

\newtheorem{thm}{Theorem}[section]

\newtheorem{prop}[thm]{Proposition}
\newtheorem{lem}[thm]{Lemma}
\newtheorem{cor}[thm]{Corollary}

\theoremstyle{definition}
\newtheorem{defn}[thm]{Definition}

\theoremstyle{remark}
\newtheorem{rem}[thm]{Remark}

\begin{document}

\title{Why do the symmetries of the monster vertex algebra form a finite simple group?}
\author{Scott Carnahan}

\begin{abstract}
Together with the 1988 construction of the monster vertex algebra $V^\natural$, Frenkel, Lepowsky, and Meurman showed that the largest sporadic simple group, known as the Fischer-Griess monster, forms the symmetry group of an infinite dimensional algebraic object whose construction was motivated by theoretical physics.  However, the fact that the symmetry group is in fact finite and simple ultimately relied on highly non-trivial group-theoretic results used in Griess's work on the monster.  We prove some properties of the automorphism group of $V^\natural$, most notably that it is is finite and simple, using recent developments in the theory of vertex operator algebras, but mostly 19th century group theory.
\end{abstract}

\maketitle

\tableofcontents

\section{Introduction}

Let us imagine a world much like our own, but where finite group theory developed much more slowly than 2-dimensional conformal field theory and the theory of vertex algebras \footnote{Alternatively, we may imagine a mathematician who, like the author, is reasonably familiar with recent developments in vertex algebras but is somewhat unclear with details about the classification of finite simple groups.}.  We might find that many finite simple groups would be discovered not as part of a classification program, but instead appear naturally from the symmetries of vertex algebras that are built with standard techniques.  For example, most of (and possibly all of) the groups of Lie type could be found as ``small index subquotients'' in the automorphism groups of easily constructed vertex algebras over finite fields \cite{GL14}, and we would find the monster simple group and its sporadic subquotients naturally through the study of orbifolds of rank 24 lattice vertex algebras.  

In this paper, we describe the fundamentals of the construction of the monster vertex algebra $V^\natural$ from \cite{FLM88}, and prove many key properties of its automorphism group, in particular that it is finite and simple, using recent results in vertex algebra theory but very little finite group theory or explicit calculation.  That is, we will play a game, where we wrestle with an unwieldy opponent ($\Aut V^\natural$) with our strongest hand (20th century finite group theory) tied behind our back.  Thus, our main group-theoretic tools for proving finiteness and simplicity are Sylow's theorem and the Frattini argument, both from the late 19th century, and we also use a basic fact from algebraic geometry, namely that a complex algebraic group with trivial Lie algebra is finite.  More precisely, our use of post-19th-century finite group theory is limited to properties of double covers of lattices, the Golay code, and the Leech lattice.

Our main result, that the automorphism group of the monster vertex algebra $V^\natural$ is finite and simple, is clearly not new: it was first proved in \cite{FLM88}, some pages after the construction of $V^\natural$ itself.  However, the original proof relies on an identification with the monster simple group $\bM$, constructed in \cite{G82}, where the fact that $\bM$ is finite and simple was proved using a lot of the machinery of the classification of finite simple groups.  This path to proof has had some simplifications made: a simpler proof that $\bM$ is finite and simple is sketched in \cite{T84}, and a simplified construction of $\bM$ is given in \cite{C85}.  Furthermore, there are now many alternative constructions of $V^\natural$, e.g., in \cite{M97}, \cite{CLS16}, \cite{ALY17}, \cite{C17}.  However, until now, these simplifications have not been substantially combined.

We take ideas from several treatments of both the monster and $V^\natural$.  The construction of $V^\natural$ is fundamentally the same as in the original work \cite{FLM88}, but the analysis of the vertex algebra structure and symmetry is simplified by recent advances in the theory of holomorphic orbifolds.  We replace a brutally explicit treatment of triality with a more intrinsically symmetric method inspired by Conway's construction of the monster in \cite{C85}, but we avoid use of finite nonassociative structures like the Parker loop.  For the proofs of finiteness and simplicity, we originally used the outline from \cite{T84}, but we found that the vertex-algebraic description makes our proofs a small fraction of the original length, and in particular allowed us to avoid the explicit computations.  We conclude by describing some additional facts about $\Aut V^\natural$ that are accessible with vertex algebraic techniques.

It is interesting to compare our results with those in \cite{N01}, where structural features of the monster were worked out assuming the generalized moonshine conjecture, but not using vertex algebras.  We are able to derive most of the results given by Norton, for example the sharp upper bounds for the exponents of various primes in the order of the monster, but we are forced to use slightly different techniques.  This is because Norton assumes a particularly strong form of generalized moonshine that has not been proved \cite{C12, CM16}.

One notable omission from this paper is the rich theory of Virasoro minimal models as a way to do explicit calculations.  There is significant work in the analysis of $V^\natural$ and the monster using these tools, in particular the existence of a vertex operator subalgebra of the form $L(1/2,0)^{\otimes 48}$ (\cite{DMZ94}) and consequent results concerning symmetries (e.g., \cite{M97}, \cite{S06}, \cite{S10}, \cite{GL11}) and constructions employing codes (e.g., \cite{DGH97}, \cite{M96} for the early theory and \cite{LYY04}, \cite{LYY05} for later developments).  

\subsection{Acknowledgements}

I would like to thank Jacob Lurie and Akshay Venkatesh for inviting me to talk at the IAS moonshine seminar in March 2021.  My preparations for the talk directly led to the writing of this paper.  I would also like to thank John Duncan and Theo Johnson-Freyd for helpful comments during the talk that led to improvements in the exposition.  Finally, I thank Hiroki Shimakura and Naoki Chigira for helpful comments and corrections on an earlier version of this paper.

\section{Two Niemeier lattices}

Before we begin our treatment of vertex algebras, we need to briefly consider lattices.

\begin{defn}
A \textbf{lattice} is a free abelian group of finite rank equipped with an integer-valued bilinear form.  A lattice $L$ is \textbf{even} if $(a,a) \in 2\bZ$ for all $a \in \bZ$, \textbf{nondegenerate} if $(a,b) = 0$ for all $a \in L$ implies $b = 0$, and \textbf{positive definite} if $(a,a) \leq 0$ implies $a=0$.  The \textbf{dual lattice} of a nondegenerate lattice $L$ is the set of all $b \in L \otimes \bQ$ such that $(a,b) \in \bZ$ for all $a \in L$.  A lattice $L$ is \textbf{unimodular} if it is equal to its dual lattice.  A \textbf{root} in a positive definite even lattice is a vector $a$ of norm 2, i.e., $(a,a) = 2$.
\end{defn}

The first nonzero example of an even positive definite lattice is $\sqrt{2}\bZ$, also known as the $A_1$ root lattice, generated by a vector $a$ satisfying $(a,a) = 2$.  This is not unimodular, since the dual lattice is $\frac{1}{\sqrt{2}}\bZ$, which contains $\sqrt{2}\bZ$ as an index 2 subgroup.

It is well-known that the rank of any positive definite even unimodular lattice is an integer multiple of 8.  The first nontrivial example is the $E_8$ lattice, in rank 8.  There are 2 isomorphism types in rank 16, 24 in rank 24 (classified by Niemeier in \cite{N73}), and more than $10^9$ in rank 32 \cite{K00}.  One can see that there is a phase transition around rank 24, beyond which lattices become difficult to classify.

A standard invariant of lattices is given by the generating function that enumerates vectors of a given length:

\begin{defn}
Let $L$ be a positive definite even lattice.  Then, the theta function of $L$ is:
\[ \theta_L(\tau) = \sum_{a \in L} q^{(a,a)/2} \]
where $\tau$ describes a point in the complex upper half-plane, and $q = e^{2\pi i \tau}$.
\end{defn}

We note that this sum converges normally to a holomorphic function on the upper half-plane, because there is a degree $d$ polynomial in $n$ that bounds the coefficient of $q^n$ for all $n$.  A short exercise in Poisson summation yields:

\begin{lem}
If $L$ is a positive definite even unimodular lattice of rank $d$, then $\theta_L$ is a modular form of weight $\frac{d}{2}$.  In particular, $\theta_L(\frac{-1}{\tau}) = \tau^{d/2} \theta_L(\tau)$.
\end{lem}

In order to study the monster, we consider two rank 24 even unimodular lattices, called Niemeier lattices: one called $N(A_1^{24})$, because its root lattice is a product of 24 copies of the $A_1$ lattice $\sqrt{2}\bZ$, and the Leech lattice $\Lambda$, which is distinguished by the property that it has no roots.

\subsection{The Golay code construction}

Let us suppose there is a positive definite even unimodular lattice $L$ of rank 24 with root system $A_1^{24}$.  Then, the sublattice $(\sqrt{2}\bZ)^{24}$ generated by roots has index $2^{24}$ in its dual lattice $(\frac{1}{\sqrt{2}}\bZ)^{24}$.  Thus, the unimodular condition implies $L$ contains the root lattice as an index $2^{12}$ subgroup, i.e., $L/(\sqrt{2}\bZ)^{24}$ is a subgroup $\cC$ of order $2^{12}$ in $(\frac{1}{\sqrt{2}}\bZ)^{24}/(\sqrt{2}\bZ)^{24} \cong (\bZ/2\bZ)^{24}$.  We may identify elements of this quotient group with subsets of the distinguished basis $\{e_i = (0,\ldots,0,\frac{1}{\sqrt{2}},0,\ldots,0)\}$ of $(\frac{1}{\sqrt{2}}\bZ)^{24}$, and we find that for $L$ to be even, it is necessary and sufficient that all elements of $\cC$ correspond to subsets whose size is a multiple of 4.  For $L$ to contain no additional roots beyond the $A_1^{24}$ root system, it is necessary and sufficient that $\cC$ have no elements corresponding to subsets of size 4.  In the language of codes, $\cC$ is a doubly even binary code of length 24 with minimal weight 8.  The following theorem then implies the lattice $L$ exists, and is uniquely determined up to isomorphism by the conditions we specified.

\begin{thm}
There exists a doubly even binary code $\cC$ of length 24 with minimal weight 8 and order $2^{12}$, and it is unique up to permutation of the basis.  The group of symmetries of this code is the sporadic simple group $M_{24}$. 
\end{thm}

This code is called the extended Golay code.  Perhaps the most concise construction of this code takes as a basis the set $\bZ/23\bZ \cup \{\infty\}$, and takes $\cC$ as the subgroup generated by the size 12 set $QR$ of quadratic residues modulo $23$ together with all linear fractional transforms of $QR$.  This construction makes a $PSL_2(\bF_{23})$ action clear, but many important properties are not so obvious.  There is an alternative construction in \cite{A94} that also yields a proof that $M_{24}$ is simple, using iterated extension of Steiner systems, starting from the projective plane over $\bF_4$.  A third construction, using the hexacode, is given in \cite{CS82} and \cite{G98}.

\begin{lem}
The theta function of $N(A_1^{24})$ is the unique weight 12 modular form with constant term $1$ and $q$-coefficient $48$: $\theta_{N(A_1^{24})}(\tau) = 1 + 48q + 195408q^2 + \cdots$.  In particular, there are 48 vectors of norm 2 and 195408 vectors of norm 4.
\end{lem}

The automorphisms of $N(A_1^{24})$ form a group of type $2^{24}.M_{24}$, where the notation means there is a normal subgroup isomorphic to $(\bZ/2\bZ)^{24}$, with quotient isomorphic to $M_{24}$.  The normal 2-group is the Weyl group, generated by root reflections, and the quotient $M_{24}$ 
comes from basis permutations that preserve the Golay code.

\subsection{The Leech Lattice}

There are many constructions of the Leech lattice $\Lambda$, but a construction starting from $N(A_1^{24})$ is most convenient for us.  The vector $\frac{1}{4}\alpha_{\cC} = \frac{1}{\sqrt{8}}(1,\ldots,1) \in \bR^{24}$ has inner product $\pm \frac{1}{2}$ with all roots in our presentation of $N(A_1^{24})$, and in fact, the inner product with any element of $N(A_1^{24})$ lies in $\frac{1}{2}\bZ$, so the sublattice $\Lambda_0$ whose vectors have integer inner product with $\frac{1}{4}\alpha_{\cC}$ has index 2, and has no roots.  $\Lambda_0$ has 3 non-identity cosets in its dual lattice, and only two of them yield even lattices, so we may construct $\Lambda$ as the unique even lattice strictly containing $\Lambda_0$ that is not $N(A_1^{24})$.  This construction realizes $\Lambda$ and $N(A_1^{24})$ as ``neighbors'' in the sense of Kneser.

The nontrivial coset contains the vector $(-\frac{3}{\sqrt{8}},\frac{1}{\sqrt{8}},\ldots,\frac{1}{\sqrt{8}})$.  More intrinsically, we may present $\Lambda$ as the lattice with vectors of the form $(a_1,\ldots,a_{24}) \in (\frac{1}{\sqrt{8}}\bZ)^{24}$ satisfying: 
\begin{enumerate}
\item $\sqrt{8} \sum_i a_i = 4m$ for some integer $m$. 
\item The set $S = \{ i | \sqrt{8} a_i \not\equiv m \mod 4\bZ \} \in \cC$, and furthermore, $\sqrt{8}a_i \equiv m \mod 2\bZ$ for all $i \in S$.  
\end{enumerate}

\begin{lem}
The theta function of $\Lambda$ is the unique weight 12 modular form with constant term $1$ and vanishing $q$-coefficient: $\theta_\Lambda(\tau) = 1 + 196560q^2 + \cdots$.  In particular, there are no vectors of norm 2 and 196560 vectors of norm 4.
\end{lem}

It is not difficult to compute the coefficients of a spanning set of the space of weight 12 modular forms, e.g., $E_4(\tau)^3 = (1 + 240\sum_{n\geq 1} q^n \sum_{d|n} d^3)^3$ and $\Delta(\tau) = q\prod_{n\geq 1}(1-q^n)^{24}$.  Thus, it is also not difficult to enumerate the vectors in $\Lambda$.

\begin{defn}
We write $Co_0$ to denote the automorphism group of $\Lambda$, and $Co_1$ to denote the quotient of $Co_0$ by its center $\pm 1$.
\end{defn}

\begin{thm} (\cite{C69})
The group $Co_1$ is simple of order $2^{21}\cdot 3^9 \cdot 5^4 \cdot 7^2 \cdot 11 \cdot 13 \cdot 23$.
\end{thm}

We briefly outline the elements of Conway's proof - perhaps the most sophisticated group theory lies in the use of Sylow's theorem and the Frattini argument, both from the late 19th century.  The order can be computed by enumerating ``frames'' or ``crosses'', namely systems of 24 mutually orthogonal lines spanned by norm 8 vectors.  The stabilizer of a cross in $Co_1$ is a maximal non-normal subgroup $N$ of type $2^{11}.M_{24}$, where the normal 2-subgroup is given by products of reflections parametrized by Golay codewords, and $M_{24}$ permutes the lines.  To get the order, one then shows that $Co_1$ acts transitively on crosses, and that any element of $Co_1$ either stabilizes a given cross or takes it to a disjoint cross - this follows from the fact that the norm 8 vectors in a cross are precisely the norm 8 vectors in a single coset of $2\Lambda$ in $\Lambda$.

Conway's proof of simplicity uses the fact that $Co_1$ acts primitively on lines spanned by norm 4 vectors, so any nontrivial normal subgroup $H$ acts transitively on these lines, and hence has order dividing 13. By Frattini, $Co_1$ is a product of $H$ and the normalizer of a 13-Sylow subgroup, so at least one of these two groups has order dividing 23.  In the latter case, we get an element of order $13 \cdot 23$, contradicting a rank restriction coming from rational canonical form.  In the former case, where $|H|\in 23\bZ$, one notes that $H \cap N$ is normal in $N$ with an element of order 23, hence all of $N$, so by maximality of $N$ we have $H = Co_1$.  

\section{Vertex algebraic ingredients}

\subsection{Basic definitions}

\begin{defn}
Let $V$ be a complex vector space.  A \textbf{field} on $V$ is a linear map $V \to V((z))$, where $V((z))$ denotes the space of formal Laurent series with coefficients in $V$.  We will write Laurent series as $\sum_{n \in \bZ} v_n z^{-n-1}$, where $v_n = 0$ for $n$ sufficiently large, and we will write fields as series $A(z) = \sum_{n \in \bZ} A_n z^{-n-1} \in (\End V)[[z,z^{-1}]]$.
\end{defn}

This notion of field is an algebraist's version of ``quantum field on $\bC \setminus\{0\}$'' in the following sense: Recall that quantum fields are axiomatized by Wightman as ``operator-valued distributions on spacetime'', which are linear maps that take in a smooth test function and produce an operator on a Hilbert space.  Here, our test functions are Laurent polynomials $f(z) = a_{-m}z^{-m} + \cdots + a_n z^n$, and the pairing with a field $A(z)$ is given by taking the residue at 0 of $f(z)A(z)dz$, i.e., the coefficient $\sum_{i=-m}^n a_i A_i$ of $z^{-1}dz$.

Vertex algebras can be viewed as an axiomatization of the vague notion of ``commutative ring of quantum fields'' where commutativity is weakened to ``locality'' in order to allow for singularities.

\begin{defn}
Two fields $A(z)$ and $B(z)$ on $V$ are \textbf{mutually local} if there there is a non-negative integer $N$ such that $(z-w)^N [A(z),B(w)] = 0$ as an element of $(\End V) [[z,z^{-1},w,w^{-1}]]$.  The smallest such $N$ is called the \textbf{order of locality} of $A(z)$ and $B(z)$.
\end{defn}

When we are given fields with no singularities, locality is the same as commutativity.  For example, in a commutative ring $R$, for any $a \in R$, the ``left multiplication by $a$'' map can be viewed as a constant field $b \mapsto ab = a_{-1}b z^0$ on $R$, and any two left-multiplication maps strictly commute.  The identity element yields the identity field $I(z) = \id_R z^0$.  However, fields in general cannot be composed in the way that left-multiplication composes in a commutative ring.  We instead have a collection of composition laws parametrized by integers.

\begin{defn}
Given a pair $A(z),B(z)$ of fields, the $m$th \textbf{residue product} is defined as:
\[ \begin{aligned} A(z)_mB(z) &= \underset{y=0}{\Res} \, A(y)B(z)(y-z)^m dy - \underset{y=0}{\Res} \, B(z)A(y)(y-z)^m dy \\
&= \sum_{n \in \bZ} \sum_{i \geq 0} (-1)^i \binom{m}{i} (A_{m-i}B_{n+i} - (-1)^mB_{m+n-i}A_i)z^{-n-1}
\end{aligned} \]
Here, we expand $(y-z)^m$ as a formal power series by the ``first field variable is big'' convention: when attached to $A(y)B(z)$, we set $|y|>|z|$, yielding the expansion $y^m (1-z/y)^m = y^m \sum_{i \geq 0} \binom{m}{i} z^i y^{-i}$, and when attached to $B(z)A(y)$, we set $|z| > |y|$, yielding the expansion $(-1)^m z^m(1- y/z)^m = (-1)^m z^m \sum_{i \geq 0} \binom{m}{i} z^{-i} y^i$.
\end{defn}

Residue products look somewhat cumbersome at first, but they satisfy the following convenient properties (see sections 1 and 2 of \cite{MN97}):
\begin{thm} \label{thm:Matsuo-Nagatomo}
\begin{enumerate}
\item Residue products with the identity field satisfy $A(z)_nI(z) = 0$ for $n\geq 0$, and $A(z)_{-1} I(z) = A(z)$.
\item If $A$ and $B$ are mutually local fields on a complex vector space, then $A(z)_nB(z) = 0$ for sufficiently large positive $n$.
\item If $A,B,C$ are mutually local fields on a complex vector space, then $A$ is mutually local with $B(z)_nC(z)$ for all $n \in \bZ$.
\item If $A,B,C$ are mutually local fields on a complex vector space, then for any $p,q,r \in \bZ$, 
\[ \begin{aligned}
\sum_{i\geq 0} & \binom{p}{i} \left(A(z)_{r+i}B(z)\right)_{p+q-i}C(z) \\
&= \sum_{i \geq 0} (-1)^i \binom{r}{i} \left( A(z)_{p+r-i}(B(z)_{q+i}C(z)) \right.\\
&\qquad \left. -(-1)^r B(z)_{q+r-i}(A(z)_{p+i}C(z)) \right).
\end{aligned} \]
\end{enumerate}
\end{thm}
The last is called ``Borcherds's identity for residue products''.

We now consider an algebraic structure that encodes the properties of a suitable collection of mutually local fields.

\begin{thm} \label{thm:Li} (\cite{L94})
Suppose we are given a collection $V$ of mutually local fields on a complex vector space $M$ that is closed under residue products and $\bC$-linear combinations, and also contains the identity field $I(z)$.  Then, $V$ has the following properties:
\begin{enumerate}
\item $V$ is a complex vector space equipped with a homogeneous $\bC$-linear multiplication map $Y: V \otimes V \to V((x))$, written $A(z) \otimes B(z) \mapsto Y(A(z),x)B(z) = \sum_{n \in \bZ} A(z)_n B(z) x^{-n-1}$.
\item $Y(A(z),x)I(z) = \sum_{n < 0} \partial^{(-n-1)}A(z) x^{-n-1}$.
\item For any $A(z),B(z) \in W$, the fields $Y(A(z),x)$ and $Y(B(z),x)$ on $V$ are mutually local and satisfy $Y(A(z),x)_n Y(B(z),x) = Y(A(z)_nB(z),x)$ for all $n \in \bZ$.
\end{enumerate}
\end{thm}
\begin{proof}
The vector space structure is clear, and the fact that $Y$ takes elements of $V$ to fields on $V$ is the second claim in Theorem \ref{thm:Matsuo-Nagatomo}.

The claim about products with the identity field is the first claim in Theorem \ref{thm:Matsuo-Nagatomo}.

Locality follows from locality for the input fields: If $A(z)$ and $B(z)$ are local of order $N$, then we may use the Borcherds identity with $r=0$, also called the commutator formula, to get
\[ \begin{aligned}
(x - y)^N&[Y(A(z),x),Y(B(z),y)]C(z) \\
&= (x-y)^N\sum_{m,n \in \bZ} [A(z)_m, B(z)_n] C(z) x^{-m-1}y^{-n-1} \\
&= (x-y)^N\sum_{m,n \in \bZ} \sum_{i =0}^{N-1} \binom{m}{i}(A(z)_iB(z))_{m+n-i}C(z) x^{-m-1}y^{-n-1} \\
&= (x-y)^N\sum_{i =0}^{N-1}\partial^{(i)}_y\delta(x-y) \sum_{j \in \bZ} (A(z)_iB(z))_j C(z) y^{-j-1}\\
&= 0
\end{aligned} \]
where we define $\partial^{(i)}_y \delta(x-y) = \sum_{m \in \bZ} \binom{m}{i}x^{-m-1}y^{m-i}$, which is known to satisfy $(x-y)^N\partial^{(i)}_y \delta(x-y) =0$ for $i \leq N$.

Finally, the residue product identity follows from the Borcherds identity with $(p,q,r) = (0,n,m)$ (also known as the associativity rule): For any $C(z) \in V$, 
\[ \begin{aligned}
Y(A(z)_m B(z),x)C(z) &= \sum_{n \in \bZ} (A(z)_m B(z))_nC(z) x^{-n-1} \\
&= \sum_{n \in \bZ} \sum_{i \geq 0} (-1)^i \binom{m}{i}(A(z)_{m-i}(B(z)_{n+i}C(z)) \\
&\qquad - (-1)^{m-p(A)p(B)}B(z)_{m+n-i}(A(z)_iC(z)))x^{-n-1} \\
&= (Y(A(z),x)_n Y(B(z),x)) C(z)
\end{aligned}\]
\end{proof}

We give this structure a name (but this is not the historical motivation):

\begin{defn}
A \textbf{vertex algebra} (over $\bC$) is a complex vector space $V$ equipped with a multiplication map $Y: V \otimes V \to V((z))$, written $u \otimes v \mapsto Y(u,z)v = \sum_{n \in \bZ} u_n v z^{-n-1}$, and a distinguished unit vector $\unit$, satisfying the following properties:
\begin{enumerate}
\item (unit axiom) $Y(u,z)\unit \in u + zV[[z]]$ for all $u \in V$.  That is, $u_n \unit$ vanishes for $n \geq 0$ and $u_{-1}\unit = u$.
\item (locality) For any $u, v \in V$, the fields $Y(u,z)$ and $Y(v,z)$ are mutually local.
\item (associativity) For any $u, v \in V$ and $n \in \bZ$, $Y(u_n v, z) = Y(u,z)_n Y(v,z)$.
\end{enumerate}
The power series $Y(u,z) \in (\End V)[[z,z^{-1}]]$ is called the ``field attached to $u$''.
\end{defn}

Because of Li's theorem, if we are given a vertex algebra $V$, then the map $Y$ gives an isomorphism from $V$ to the vertex algebra of fields on $V$ of the form $Y(u,z)$ for $u \in V$, and the inverse isomorphism takes a field $A(z)$ to the constant term at identity $A(z)\unit|_{z=0}$.  In particular, by combining the Borcherds identity for residue products with $Y(u,z)_nY(v,z) = Y(u_n v,z)$, we find that the ``Borcherds identity'' (often called the Jacobi identity when put into power series form) holds for all $u,v,w \in V$ and $p,q,r \in \bZ$:

\[ \begin{aligned}
\sum_{i\geq 0} & \binom{p}{i} \left(u_{r+i}v\right)_{p+q-i}w \\
&= \sum_{i \geq 0} (-1)^i \binom{r}{i} \left( u_{p+r-i}(v_{q+i}w) \right.\\
&\qquad \left. -(-1)^r v_{q+r-i}(u_{p+i}w) \right).
\end{aligned} \]

\begin{rem}
This is not the only way vertex algebras are axiomatized: One can replace the last two axioms with the Borcherds identity (also known as the Jacobi identity when put into a power series form), and one may replace the condition $Y(u,z)_nY(v,z) = Y(u_n v,z)$ with either the straightforwardly equivalent rule $(a_m b)_n c = \sum_{ i \geq 0} (-1)^i\binom{m}{i}(a_{m-i}(b_{n+i}c)) - (-1)^mb_{m+n-i}(a_i c)$, or translation-equivariance $Y(Tu,z) = [T,Y(u,z)] = \partial_z Y(u,z)$.
\end{rem}

By making various substitutions to the Jacobi identity, one obtains many results, such as:
\begin{enumerate}
\item (commutator formula) $[u_p,v_q] = \sum_{i \geq 0} \binom{p}{i}(u_i v)_{p+q-i}$: set $r=0$.
\item (skew symmetry) $Y(v,z)u = e^{Tz}Y(u,-z)v$, where $Tu = u_{-2}\unit$: set $w = \unit$, $p=-1$, $q=0$.
\item (weak associativity) For any $u,v,w \in V$, there is some $N>0$ such that $(y+z)^NY(u,y+z)Y(v,z)w = (y+z)^N Y(Y(u,y)v,z)w$: let $N$ satisfy $u_p w = 0$ for all $p>N$.
\end{enumerate}
This convenience is one reason many treatments of vertex algebras use the Jacobi identity in place of locality.

The simplest examples of vertex algebras are commutative rings, where we set $Y(u,z)v = u_{-1}v z^0 = uv$.  Slightly more generally, if we are given a unital commutative associative algebra with derivation $D$, we may set $Y(u,z)v = (e^{Dz}u)v$ to get a vertex algebra.  These vertex algebras satisfy the property that $Y(u,z)Y(v,w) = Y(v,w)Y(u,z)$ for all $u,v \in V$.  Vertex algebras satisfying this property are called ``commutative'', and any commutative vertex algebra arises from a commutative associative algebra with derivation this way.  Commutativity is equivalent to the property that $Y(u,z)v \in V[[z]]$, i.e., the power series given by products have no poles at $z=0$.  We may therefore think of a vertex algebra as a generalization of a commutative ring with derivation, where the products are allowed to have singularities at $z=0$.

Before considering the non-commutative examples of interest, we introduce a convenient theorem for showing a collection of fields generate a vertex algebra.

\begin{thm} (Reconstruction theorem, Proposition 3.1 of \cite{FKRW95})
Let $V$ be a vector space, $\unit$ an element of $V$, $T$ an endomorphism of $V$, and let $\{ A^i(z)\}_{i \in J}$ be a set of fields in $V$ including the identity field $I(z) = \id_V z^0$.  Suppose the following properties hold:
\begin{enumerate}
\item $A^i(z)\unit \in A^i + zV[[z]]$ for some vector $A^i \in V$, for each $i \in J$.
\item $T\unit = 0$ and $[T,A^i(z)] = \frac{d}{dz}A^i(z)$
\item All $A^i(z)$ are mutually local.
\item $V$ is spanned by $A^{i_1}_{n_1} \cdots A^{i_k}_{n_k}\unit$, where $k \geq 0$ and each $n_i<0$.
\end{enumerate}
Then, the assignment
$Y(A^{i_1}_{n_1} \cdots A^{i_k}_{n_k} \unit, z) = A^{i_1}(z)_{n_1} \cdots A^{i_k}(z)_{n_k}I(z)$
defines a vertex algebra structure on $V$, and it is the unique vertex algebra structure such that $Y(A^i_{-1} \unit,z) = A^i(z)$ for all $i \in J$.
\end{thm}

We also note that \cite{FLM88} introduced a variant of the notion of vertex algebra that has some additional structure and satisfies convenient finiteness conditions:

\begin{defn}
A \textbf{vertex operator algebra} of central charge $c \in \bC$ is a quadruple $(V,Y,\unit,\omega)$, where $(V,Y,\unit)$ is a vertex algebra, and $\omega \in V$ is a \textbf{conformal vector} satisfying the following properties:
\begin{enumerate}
\item The coefficients of the field $T(z) = Y(\omega,z) = \sum L_n z^{-n-2}$ satisfy $[L_m,L_n] = (m-n)L_{m+n} + \binom{m+1}{3}\delta_{m,-n} \frac{c}{2}\id_V$, i.e., $V$ is endowed with the structure of a representation of the Virasoro algebra with central charge $c$.
\item $L_0$ acts diagonalizably with integer eigenvalues that are bounded below and finite dimensional eigenspaces, and furthermore, $L_0 \unit = 0$ and $L_0 \omega = 2\omega$.
\item $L_{-1}u = u_{-2}\unit$ for all $u \in V$.
\end{enumerate}
\end{defn}

All of the vertex algebras we consider are vertex operator algebras for a natural choice of $\omega$.  Sometimes, there are alternative choices of conformal vector, but this is not important in our treatment.

\subsection{Lattice vertex algebras}

One of the easiest noncommutative examples of a vertex algebra is the Heisenberg vertex algebra attached to a vector space $\fh$ with inner product.  We form the infinite dimensional Heisenberg Lie algebra $\hat{\fh}$, which decomposes into $\fh[t,t^{-1}] \oplus \bC K$ as a vector space, with Lie bracket determined by $[u t^m, vt^n] = m(u,v)\delta_{m,-n}K$, $[K,ut^n] = 0$ for all $u, v \in \fh$, $m,n \in \bZ$.  This is a central extension of the abelian Lie algebra $\fh[t,t^{-1}]$.  We let $\fh^+$ denote the Lie subalgebra $\fh[t] \oplus \bC K$, and for each $u \in \fh$, we let $\bC_u$ denote the one-dimensional representation on which $vt^n$ acts by $(v,u)\delta_{n,0}$ and $K$ acts by identity.  Then, we define $\pi_u$ (or $\pi^\fh_u$ if $\fh$ is ambiguous) to be the induced representation $\Ind_{\fh^+}^{\hat{\fh}} \bC_u$.  Because $\hat{\fh}$ decomposes as a direct sum of abelian subalgebras $\fh^+$ and $t^{-1}\fh[t^{-1}]$, we see by the Poincar\'e-Birkhoff-Witt theorem that $\pi_u$ is a free module of rank 1 for the universal enveloping algebra $U(t^{-1}\fh[t^{-1}])$, which is isomorphic to the symmetric algebra $\Sym(t^{-1}\fh[t^{-1}])$.  In particular, given an ordered basis $e^1,\ldots,e^d$ of $\fh$, we obtain a basis of $\pi_u$ given by monomials of the form $e^{i_1}t^{n_1}\cdots e^{i_k}t^{n_k} v_u$, where $n_1 \leq \cdots \leq n_k < 0$ and $i_j < i_{j+1}$ if $n_j = n_{j+1}$.  

\begin{thm}
There is a unique vertex algebra structure on $\pi_0$ satisfying
\[ \unit = v_0 \quad \text{and} \quad Y(u t^{-1},z) = \sum_{n \in \bZ} u t^n z^{-n-1}, \]
for all $u \in \fh$.  Furthermore, the vector $\omega = \frac{1}{2}\sum_{i=1}^d (e^i t^{-1})^2\unit$ endows $\pi_0$ with a vertex operator algebra structure with central charge $d$, where $e^1,\ldots,e^d$ is any orthonormal basis of $\fh$
\end{thm}

The first claim is immediate from the reconstruction theorem, once we establish pairwise locality of the fields $Y(e^i t^{-1},z)$.  The essential relation $(z-w)^2[Y(e^i t^{-1},z)Y(e^j t^{-1},w)] = 0$ follows from a straightforward calculation using the Lie bracket on $\hat{\fh}$.  The grading on $\pi_0$ by $L_0$-eigenvalues is non-negative, and when the inner product on $\fh$ is nonsingular, the degree $n$ part has dimension $p(n)$, the number of partitions of $n$ into positive integers.  We can assemble the graded dimensions into a power series, called the character:

\begin{lem}
The character of $\pi_0$, given by $\sum_{n \geq 0} \dim (\pi_0)_n q^{n-d/24}$, is equal to the modular form
\[ \eta(\tau)^{-d} = q^{-d/24}\prod_{k=1}^{\infty} (1-q^k)^{-d} \]
\end{lem}

The extra $-d/24$ in the definition, called the ``vacuum energy shift'' by physicists, is what we need to get a modular form.

We briefly note that the group of vertex algebra automorphisms of $\pi_0$ is the orthogonal group of $\fh$.  The action on the generating vectors $\{ ut^{-1}| u \in \fh\}$ extends naturally to $\Sym(t^{-1}\fh[t^{-1}])$.

We now move on to the lattice construction.  Let $L$ be an even integral lattice spanning a vector space $\fh$.  Then, we wish to produce a vertex algebra structure on the direct sum $\bigoplus_{u \in L} \pi_u$ extending that of $\pi_0$.  There is a degenerate way to do this, where the multiplication map vanishes when restricted to $\pi_u \otimes \pi_v$ for $u,v \neq 0$.  However, the situation where this vanishing never happens is substantially more interesting.

\begin{thm} (\cite{B86}, \cite{FLM88})
Given an even integral lattice $L$ spanning a vector space $\fh$, there exists a unique (up to isomorphism) vertex algebra structure on the direct sum $\bigoplus_{u \in L} \pi_u$, subject to the condition that the multiplication map is nonzero when restricted to $\pi_u \otimes \pi_v$ for any $u,v \in L$.  Furthermore, if $L$ is positive definite, then the vertex operator algebra structure on $\pi_0$ extends to a vertex operator algebra structure on the direct sum.
\end{thm}

A proof of this theorem using reconstruction can be found in Chapter 5 of \cite{FBZ04}.  It suffices to define a suitable field $Y(v_u,z)$ for the generating vector $v_u$ of each irreducible Heisenberg module $\pi_u$.  However, such a field is uniquely determined up to a scalar by the requirements that $Y(v_u,z)$ take $\pi_v$ to $\pi_{v+u}((z))$ for all $v \in L$, and that $Y(v_u,z)$ be mutually local with the fields from $\pi_0$.  We will not describe the field in detail here, but we remark that locality imposes a nontrivial condition on the undetermined scalars: We may choose to take them from $\pm 1$, but we have to choose signs in a way that corresponds to constructing a central extension $\hat{L}$ of the lattice by the group $\{ \pm 1\}$.  We will write $V_L$ for the lattice vertex algebra.



\begin{prop}
Let $L$ be an even integral lattice of rank $d$.  Then, $V_L$ is $\bZ$-graded with weight given by eigenvalues of $L_0$.  If $L$ is positive definite, then the weights are non-negative, and the character of $V_L$, defined as $\sum_{n \geq 0} \dim (V_L)_n q^{n-d/24}$, is equal to the modular function $\frac{\theta_L(\tau)}{\eta(\tau)^d}$, for $\tau$ in the complex upper half plane, where $q = e^{2\pi i \tau}$.  Here, $\theta_L(\tau) =  \sum_{\alpha \in L} q^{(\alpha,\alpha)/2}$ is a modular form of weight $d/2$, and $\eta(\tau) = q^{1/24}\prod_{n \geq 1} (1-q^n)$ is a modular form of weight $1/2$.
\end{prop}

We briefly consider automorphism groups of lattice vertex algebras, and here, we must restrict our view to positive definite lattices to ensure the groups are finite dimensional.  The reason is that if $h$ is a weight 1 vector, then $h_0$ is a derivation, i.e., an infinitesimal automorphism, and the weight 1 space is finite dimensional if and only if the lattice is positive definite.  

\begin{thm} (\cite{DN99}) \label{thm:automorphism-group-of-lattice-voa}
Let $L$ be a positive definite even lattice, and let $V_L$ denote the vertex operator algebra attached to $L$.  Then, the automorphism group of $V_L$ is a complex Lie group of the form $GO$, where:
\begin{enumerate}
\item $G$ is the connected component of identity, which is in particular a normal subgroup.  $G$ has a natural structure of a complex Lie group whose Lie algebra is isomorphic to the weight 1 subspace of $V_L$, with Lie bracket $[u,v] = u_0 v$.  That is, $G$ is generated by $e^{h_0}$ for $h$ in the weight 1 subspace.
\item $O$ is the finite group of orthogonal transformations of a double cover $\pi:\hat{L} \to L$.  The precise structure of the double cover is uniquely determined by the commutator relation $[a,b] = (-1)^{(\pi(a),\pi(b))}$ for all $a,b \in \hat{L}$.  The group $O$ lies in an exact sequence $1 \to \Hom(L,\pm 1) \to O \to \Aut L \to 1$.
\item The intersection of $G$ and $O$ contains a copy of $\Hom(L, \pm 1)$ lying in the distinguished maximal torus (generated by $\{e^{h_0}, h \in \pi_0\}$) of $G$.
\end{enumerate}
\end{thm}


As we mentioned earlier, our main examples for the lattice $L$ are the following positive definite even unimodular lattices of rank 24:
\begin{enumerate}
\item The lattice $N(A_1^{24})$ with root system $A_1^{24}$.
\item The Leech lattice $\Lambda$, with no roots. 
\end{enumerate}

\begin{thm}
The automorphism group of $V_\Lambda$ has the form $(\bC^\times)^{24}.Co_0$.
\end{thm}
\begin{proof}
This follows from Theorem \ref{thm:automorphism-group-of-lattice-voa} as follows: the connected component of identity $G$ is infinitesimally generated by the weight 1 space, which is the rank 24 abelian Lie algebra $\Lambda \otimes \bC$.  $G$ therefore acts as scalars on the $\Lambda$-graded Heisenberg modules by the identification $\Hom(\Lambda, \bC^\times) \cong (\bC^\times)^{24}$.  The subgroup $O$ is a $2^{24}$-cover of $\Aut \Lambda = Co_0$, and $G \cap O$ is precisely $\Hom(\Lambda,\pm 1)$, which is identified with the $2^{24}$ subgroup.  Thus, $(\Aut V_\Lambda)/G \cong Co_0$.
\end{proof}

\begin{thm}
The automorphism group of the vertex algebra of $N(A_1^{24})$ has the form $GO$, where $G$ is the quotient of $(SL_2(\bC))^{24}$ by a copy of the Golay code in the center $(\pm I_2)^{24}$, and $O$ has the form $2^{24}.2^{24}.M_{24}$.  The copy of the Golay code is precisely the one whose codewords describe the cosets of $(\sqrt{2}\bZ)^{24}$ in $N(A_1^{24})$.

\end{thm}
\begin{proof}
Because the root system of $N(A_1^{24})$ is $A_1^{24}$, the Lie algebra of $G$ is $\fsl_2^{24}$.  Thus, the connected component of identity is a quotient of $SL_2(\bC)^{24}$ by a subgroup of its center.

Given $u = (u^1,\ldots,u^{24}) \in (\frac{1}{\sqrt{2}}\bZ)^{24}$, a torus element $t = (t_1,\ldots,t_{24}) \in (\bC^\times)^{24}$ acts on $\pi_u$ by the scalar $\prod_{a=1}^{24} t_a^{\sqrt{2}u^a}$.  If we specify $t_i \in \pm 1$, we find that the scalar is uniformly identity if and only if we have an even number of $-1$s for each Golay codeword.  Equivalently, the subset of $-1$ indices in $\{1,\ldots,24\}$ is orthogonal to $\cC$.  Because $\cC$ is self-orthogonal, it is necessary and sufficient that this subset is an element of $\cC$ itself.  This proves the claim about $G$.

$O$ is the orthogonal group of the double cover of $N(A_1^{24})$, hence it lies in an exact sequence
\[ 1 \to (\pm 1)^{24} \to O \to \Aut N(A_1^{24}) \to 1 \]
This proves the claim about $O$.
\end{proof}

\begin{defn}
Given an automorphism $\sigma$ of a lattice $L$, we say an automorphism $\tilde{\sigma}$ of $V_L$ is a \textbf{lift} of $\sigma$ if $\tilde{\sigma}$ sends vectors in $\pi_u$ to $\pi_{\sigma(u)}$ for all $u \in L$.
\end{defn}

In general the obstruction to an automorphism $g$ of order $n$ of a lattice admitting an order $n$ lift is $n$ being even and $(v,g^{n/2}v)$ being odd for some lattice vector $v$ \cite{B92}.  In particular, this obstruction vanishes for the $-1$ automorphism of an even lattice, so $-1$ always admits an order 2 lift.

\subsection{The diagonal octahedral group}

We introduce a particular subgroup of automorphisms of $V_{N(A_1^{24})}$ that is isomorphic to $S_4$, and most importantly, a pair of conjugate involutions within this group.

\begin{defn}
The \textbf{binary octahedral group} $2O$ is the following group of 48 unit quaternions:
\[ 2O = \{ \pm x, \frac{\pm 1 \pm i \pm j \pm k}{2}, \frac{\pm x \pm y}{\sqrt{2}} | x,y \in \{1,i,j,k\}, x \neq y\}. \]
We also view this group as a subgroup of $SL_2(\bC)$ by fixing once and for all the linear representation $i \mapsto \begin{pmatrix} i & 0 \\ 0 & -i \end{pmatrix}$, $j \mapsto \begin{pmatrix} 0 & i \\ i & 0 \end{pmatrix}$, $k \mapsto \begin{pmatrix} 0 & -1 \\ 1 & 0 \end{pmatrix}$.
\end{defn}

The binary octahedral group is the normalizer of the subgroup $Q_8 = \{ \pm 1, \pm i, \pm j, \pm k\}$ in the unit quaternions.  A few short calculations yield the following:

\begin{lem} \label{lem:adjoint-action-of-binary-octahedral}
Under the adjoint representation of $SL_2(\bC)$ on its Lie algebra, the quaternion $i$ acts as identity on diagonal matrices and $-1$ on the span of $E = \begin{pmatrix} 0 & 1 \\ 0 & 0 \end{pmatrix}$ and $F = \begin{pmatrix} 0 & 0 \\ 1 & 0 \end{pmatrix}$. The quaterion $j$ acts as $-1$ on diagonal matrices, and exchanges $E$ with $F$.  The images of the quaternions $i, j, k$ in $\Aut \fsl_2$ have order 2, and are the non-identity elements in a group $V_4$ of order 4.  The image of the binary octahedral group is isomorphic to $S_4$, which naturally permutes the elements of $V_4$, and yields an identification with the holomorph $V_4 \rtimes \Aut V_4$.
\end{lem}

We may alternatively introduce $2O$ by first considering the 4-group in $SO_3(\bR)$ generated by half-rotations about the $x$-, $y$-, and $z$-axes, defining the octahedral group as its normalizer (or just explicitly as symmetries of an octahedron or cube), and lifting to the double cover $SU(2)$.

\begin{defn}
We write $\rho$ to denote the representation of the binary octahedral group on the vertex operator algebra $V_{N(A_1^{24})}$ by the composite $2O \to SL_2(\bC) \to SL_2(\bC)^{24} \to \Aut V_{N(A_1^{24})}$, where the second map is the diagonal embedding.
\end{defn}

\begin{lem} \label{lem:image-of-binary-octahedral}
The image of the quaternion $-1$ under $\rho$ is identity on $V_{N(A_1^{24})}$.  In particular, the action of the binary octahedral group factors through the quotient $S_4$, and the elements $\rho(i), \rho(j),\rho(k)$ are the mutually conjugate non-identity elements in the normal 4-subgroup of this $S_4$.
\end{lem}
\begin{proof}
The quaternion $-1$ acts by $(-I_2,\ldots,-I_2) \in SL_2(\bC)^{24}$, which lies in the distinguished torus, so it acts by scalars on each $\pi_u$ for $u \in N(A_1^{24})$.  For any $u = (u^1,\ldots,u^{24}) \in (\frac{1}{\sqrt{2}}\bZ)^{24}$ the scalar is $\prod_{a=1}^{24} (-1)^{\sqrt{2}u^a}$.  For $u \in N(A_1^{24})$, this scalar is 1 by the evenness of the Golay code.  The last claim then follows immediately from Lemma \ref{lem:adjoint-action-of-binary-octahedral}.
\end{proof}

\begin{lem} \label{lem:fixed-point-for-Neimeier}
The quaternion $j$ acts on $V_{N(A_1^{24})}$ as a lift of the $-1$ automorphism on $N(A_1^{24})$.  That is, $j$ stabilizes $\pi_0$, and restricts to isomorphisms $\pi_u \to \pi_{-u}$ for each $u \in N(A_1^{24})$.  The quaternion $i$ acts as identity on the vertex subalgebra $V_{\Lambda_0}$ and acts as $-1$ on $\pi_u$ for $u \in N(A_1^{24}) \setminus \Lambda_0$.  In particular, the fixed point subalgebra $V_{N(A_1^{24})}^i$ is equal to $V_{\Lambda_0}$.
\end{lem}
\begin{proof}
By Lemma \ref{lem:adjoint-action-of-binary-octahedral}, $i$ acts as identity on the Heisenberg subalgebra, and $j$ acts as $-1$.  The claim about $j$ then follows from this property of the action.  Now, recall the action of $i$ is given by $(\begin{pmatrix} i & 0 \\ 0 & -i \end{pmatrix}, \ldots, \begin{pmatrix} i & 0 \\ 0 & -i \end{pmatrix})$, which lies in the torus so for any lattice vector $u = (u^1,\ldots,u^{24}) \in (\frac{1}{\sqrt{2}}\bZ)^{24}$, $\rho(i)$ acts as the scalar $\prod_{a=1}^{24} i^{\sqrt{2} u^a}$ on the corresponding module $\pi_u$.  The map $N(A_1^{24}) \to \pm1$ given by $u \mapsto \prod_{a=1}^{24} i^{\sqrt{2} u^a}$ has kernel equal to $\Lambda_0$.
\end{proof}

\begin{lem} \label{lem:S3-action-on-fixed-points}
The fixed point vertex subalgebra $V_{N(A_1^{24})}^{\rho(Q_8)}$ for the action of the 4-group $\rho(Q_8) = \langle \rho(i),\rho(j) \rangle$ inherits an action of $S_3$ from $V_{N(A_1^{24})}$.
\end{lem}
\begin{proof}
The action of $\rho(2O)$ on the fixed points for $\rho(Q_8)$ is clearly trivial when restricted to the normal subgroup $\rho(Q_8)$, so the action factors through the quotient.  Since $\rho(2O)$ admits a section of the quotient map, the action of this section is transported to the fixed point subalgebra.
\end{proof}

\subsection{Modules}

Vertex algebras have a notion of module that is analogous to the situation with commutative rings.

\begin{defn}
Let $V$ be a vertex algebra.  A $V$-\textbf{module} is a complex vector space $M$ together with a map $V \otimes M \to M((z))$, written $u \otimes w \mapsto Y_M(u,z)w = \sum_{n \in \bZ} u_n w z^{-n-1}$ satisfying the following conditions:
\begin{enumerate}
\item (unit) $Y_M(\unit,z)w = w = wz^0$ for all $w \in W$.
\item (weak associativity) For any $u,v \in V$ and $m \in M$, there is some $N$ such that $(z+w)^N Y_M(u,z+w)Y_M(v,w)m = Y_M(Y(u,z)v,w)m$.
\end{enumerate}
\end{defn}

We note that we get an equivalent definition if we replace the second condition with the following: For any $u,v \in V$ and $m \in M$, the operator-valued power series $Y_M(u,z)Y_M(v,w)m \in M((z))((w))$, $Y_M(v,w)Y_M(u,z)m \in M((w))((z))$ and $Y_M(Y(u,z-w)v,w)m \in M((w))((z-w))$ are expansions of a common element of $M[[z,w]][z^{-1},w^{-1},(z-w)^{-1}]$, i.e., there is some $N>0$ such that multiplying these three power series by $z^N w^N (z-w)^N$ yields equal series in $M[[z,w]]$.  This yields an equivalence between modules and split square-zero extensions of a vertex algebra.  One may also define modules using the Jacobi identity.

\begin{thm}
Let $\fh$ be a complex vector space with nondegenerate inner product.  Then, any irreducible $\pi_0$-module has the form $\pi_u$ for some $u \in \fh$, where $Y_{\pi_u}(e^i t^{-1},z) = \sum_{n \in \bZ} e^i t^n z^{-n-1}$. 
\end{thm}

While it is easy to understand the irreducible $\pi_0$-modules, the category of all $\pi_0$-modules has a complicated extension structure.  In particular, the finitely generated $\pi_0$-modules that admit a $\bZ_{\geq 0}$-grading are in natural bijection with finite length $\Sym \fh$-modules, or equivalently, finite length coherent sheaves on $\fh$.  To avoid these pathologies, we introduce complete reducibility hypotheses.

\begin{defn}
Let $V$ be a vertex operator algebra.  We say $V$ is \textbf{strongly rational} if the following conditions are satisfied:
\begin{enumerate}
\item The operator $L_0$ on $V$ has non-negative eigenvalues, and the kernel of $L_0$ is spanned by $\unit$ - this is sometimes called the ``CFT type'' condition.
\item Any $V$-module is a direct sum of irreducible $V$-modules.
\end{enumerate}
We say $V$ is \textbf{strongly holomorphic} if $V$ is not only strongly rational, but also any irreducible $V$-module is isomorphic to $V$.  That is, the representation category of a strongly holomorphic vertex operator algebra is equivalent to the category of complex vector spaces.
\end{defn}

The main examples we consider are lattices:

\begin{thm} \label{thm:classification-of-lattice-va-modules} (\cite{D93})
Let $L$ be a positive definite even lattice, and let $L^\vee$ be the dual lattice.  Then, $V_L$ is strongly rational, and moreover, the category of $V_L$-modules is equivalent to the category of $L^\vee/L$-graded vector spaces.  The irreducible $V_L$-modules have the form $V_{L+\gamma} = \bigoplus_{u \in L + \gamma} \pi_u$ for $\gamma \in L^\vee$.  In particular, $V_L$ is strongly holomorphic if and only if $L$ is unimodular.
\end{thm}

Strongly rational vertex operator algebras satisfy the following key properties:
\begin{thm}
Let $V$ be a strongly rational vertex operator algebra.  Then:
\begin{enumerate}
\item (Modular invariance) The characters of irreducible $V$-modules converge to holomorphic functions on the complex upper half-plane, and the action of $SL_2(\bZ)$ by M\"obius transformations induces a finite dimensional representation on the space of characters \cite{Z96}.
\item (Verlinde Formula) The modular $S$-transformation on characters of irreducible modules diagonalizes the fusion rule matrices for irreducible $V$-modules \cite{H04}.
\item (Modular tensor structure) The category of $V$-modules has a natural modular tensor category structure \cite{H05}.
\end{enumerate}
\end{thm}




An important case we consider is fixed points of holomorphic vertex operator algebras under actions of finite groups.  In the following theorem, the assumption that $G$ is solvable is conjectured to be superfluous, and by \cite{M19}, it can be removed when $V^G$ satisfies the ``$C_2$-cofiniteness'' condition.

\begin{thm} \label{thm:modules-for-solvable-fixed-point}
Let $V$ be a strongly holomorphic vertex operator algebra, and let $G$ be a finite solvable group of automorphisms.  Then:
\begin{enumerate}
\item The fixed point subalgebra $V^G$ is strongly rational \cite{CM16}.
\item The representation category of $V^G$ is equivalent to the category of $D^\alpha(\bC G)$-modules, for some $\alpha \in H^3(G,\bC^\times)$, where $D^\alpha(\bC G)$ is the $\alpha$-twisted Drinfeld double of the group ring $\bC G$ \cite{DNR21}.
\end{enumerate}
\end{thm}

For the cases of $V_\Lambda$ and $V_{N(A_1^{24})}$ and certain involutions, we will describe the modules more concretely later.  

Vertex algebras admit a different notion of module that doesn't appear in the usual study of associative algebras:

\begin{defn}
Let $V$ be a vertex algebra, and let $g$ be an automorphism of finite order $k$.  A $g$-\textbf{twisted} $V$-\textbf{module} is a vector space $M$ equipped with a linear map $V \otimes M \to M((z^{1/k}))$, written $u \otimes w \mapsto Y_M(u,z)w = \sum_{n \in \frac{1}{k}\bZ} u_n w z^{-n-1}$ satisfying the following conditions:
\begin{enumerate}
\item $Y_M(\unit,z)w = w = wz^0$ for all $w \in W$.
\item For any $u,v \in V$ and $m \in M$, the operator-valued power series $Y_M(u,z)Y_M(v,w)m \in M((z^{1/k}))((w^{1/k}))$, $Y_M(v,w)Y_M(u,z)m \in M((w^{1/k}))((z^{1/k}))$ and $Y_M(Y(u,z-w)v,w)m \in M((w^{1/k}))((z-w))$ are expansions of a common element of $M[[z^{1/k},w^{1/k}]][z^{-1},w^{-1},(z-w)^{-1}]$, i.e., there is some $N>0$ such that multiplying these three power series by $z^N w^N (z-w)^N$ yields equal series in $M[[z^{1/k},w^{1/k}]]$.
\item If $a \in V$ satisfies $ga = e^{2\pi i k/n}a$ for some $k \in \bZ$, then $Y_M(a,z) \in z^{k/n}(\End M)[[z,z^{-1}]]$.
\end{enumerate}
\end{defn}

In conformal field theory, modules are called ``sectors'', and are naturally attached to points on a Riemann surface.  One may then compute quantities like correlation functions by choosing vectors in the modules, and applying the global geometry of the curve.  In order to consider twisted modules, called ``twisted sectors'', one considers a Galois cover of a Riemann surface, and attaches twisted modules to ramified points, where the monodromy matches the twisting.

The main result about twisted modules that we need is:

\begin{thm} (\cite{DLM97})
Let $V$ be a strongly holomorphic vertex operator algebra, and let $g$ be a finite order automorphism of $V$.  Let $T_g(\tau) = \sum_n \Tr(g|V_n) q^{n-c/24}$ be the graded trace of $g$ on $V$, where we define $q = e^{2\pi i \tau}$ for $\tau$ on the complex upper half-plane.  Then, there exists a unique irreducible $g$-twisted $V$-module $V(g)$ up to isomorphism, and its character $\sum \dim V(g)_n q^{n-c/24}$ is a nonzero scalar multiple of $T_g(-1/\tau)$ as a holomorphic function on the complex upper half-plane.
\end{thm}

We can say more in the particular case of the $-1$-automorphism of positive definite even unimodular lattices:

\begin{thm} (\cite{FLM88})
Let $\theta$ be an order 2 lift of $-1$ on a positive definite even unimodular lattice $L$, and let $\hat{L}$ be the unique central extension of $L$ by $\pm 1$ satisfying $[\hat{a},\hat{b}] = (-1)^{(a,b)}$ for any lifts $\hat{a}, \hat{b} \in \hat{L}$ of vectors $a,b \in L$.  Then, $V_L(\theta)$ has the form $T \otimes \Sym(t^{-1/2}\fh[t^{-1}])$, where $\fh = L \otimes \bC$, and $T$ is the irreducible $\hat{L}/2L$-module of rank $2^{(\dim \fh)/2}$.
\end{thm}

Applying this to our examples, we have the following:

\begin{cor} \label{cor:characters-for-lattice-involutions}
The graded trace of $\rho(i)$ on $V_{N(A_1^{24})}$ is $\frac{2\theta_{L_0}(\tau)-\theta_{N(A_1^{24})}(\tau)}{\eta(\tau)^{24}} = \frac{\eta(\tau)^{24}}{\eta(2\tau)^{24}} = q^{-1} - 24 + 276q - 2048q^2 + \cdots$.  The irreducible $\rho(i)$-twisted $V_{N(A_1^{24})}$-module has character $2^{12}\frac{\eta(\tau)^{24}}{\eta(\tau/2)^{24}} = 4096q^{1/2} + 98304 q + \cdots$.  The same formulas hold if we replace $\rho(i)$ with $\rho(j)$.  For any order 2 lift $\theta$ of the $-1$ automorphism of $\Lambda$ to $V_\Lambda$, the graded trace of $\theta$ is $\frac{\eta(\tau)^{24}}{\eta(2\tau)^{24}}$, i.e., the same as the graded trace of $\rho(i)$.  The irreducible $\theta$-twisted $V_\Lambda$ module has character $2^{12}\frac{\eta(\tau)^{24}}{\eta(\tau/2)^{24}}$, i.e., the same as $V_{N(A_1^{24})}(\rho(i))$.  In each case, the twisted module decomposes into the direct sum of 2 irreducible modules for the fixed-point vertex subalgebra, and these submodules are the $+1$ and $-1$ eigenspaces for the action of $e^{2\pi i L_0}$.
\end{cor}


\subsection{Holomorphic orbifolds}

The monster vertex algebra was originally constructed using the following method:
\begin{enumerate}
\item Start with the Leech lattice vertex operator algebra $V_\Lambda$, which is strongly holomorphic (Theorem \ref{thm:classification-of-lattice-va-modules}).
\item Take the vertex subalgebra of fixed points under an order 2 automorphism $\theta$.
\item Adjoin an irreducible module for the fixed-point subalgebra.
\end{enumerate}

More generally, we may take fixed points for any group of automorphisms, and try to attach modules for the fixed point subalgebra.  For finite groups of automorphisms, this is called the ``orbifold construction'', since in conformal field theory, it is conjectured to produce the chiral algebra attached to a quotient orbifold. 

 Recent progress in the representation theory of fixed points has yielded the following:

\begin{thm} (\cite{EG18})
Let $V$ be a strongly holomorphic vertex operator algebra, and let $G$ be a finite solvable group of automorphisms.  Then strongly holomorphic extensions of $V^G$ are parametrized by pairs $(K,\psi)$, where $K$ is a subgroup of $G$ on which the restriction of $\omega$ is a coboundary, and $\psi \in Z^2(K,\bC^\times)$.  In particular, $V$ corresponds to the pair $(\{e\},1)$.
\end{thm}

Similar to the situation with the properties of fixed points that we mentioned earlier, the assumption that $G$ is solvable is conjecturally superfluous.  The subgroup $K$ parametrizes the set of twisted modules whose $V^G$-submodules appear in the extension.

\begin{defn}
Let $V$ be a strongly holomorphic vertex operator algebra, and let $G$ be a finite group of automorphisms.    The \textbf{orbifold dual} of $V$ with respect to $G$, written $V/\!/G$, is the strongly holomorphic extension of $V^G$ attached to the pair $(G,1)$, if it exists.  If $G$ is a cyclic group generated by an element $g \in G$, we write $V/\!/g$ instead.
\end{defn}

In the case where $G$ is a cyclic group generated by an element $g$, we can describe the orbifold dual more concretely.

\begin{thm} \label{thm:vEMS} (\cite{vEMS15})
Let $V$ be a strongly holomorphic vertex operator algebra, and let $g$ be a finite order automorphism of $V$.  The orbifold dual $V/\!/g$ exists if and only if the irreducible $g$-twisted module $V(g)$ has $L_0$-spectrum in $\frac{1}{|g|}\bZ$.  If this condition holds, then each irreducible twisted $V$-module $V(g^i)$ admits a faithful action of $\bZ/|g|\bZ$ satisfying the following properties:
\begin{enumerate}
\item The decomposition of $V(g^i)$ into eigenspaces for this action is a decomposition into $|g|$ irreducible modules for the fixed-point subalgebra $V^g$.
\item The action on $V(g)$ is given by $a \mapsto e^{2\pi i aL_0}a$, meaning in particular that the fixed-point subspace is precisely the subspace on which $L_0$ acts with integral eigenvalues.
\item $V/\!/g$ decomposes as a $V^g$-module into the direct sum of fixed points $\bigoplus_{i=0}^{|g|-1} V(g^i)^{\bZ/|g|\bZ}$.
\end{enumerate}
Furthermore, there is an action of a cyclic group $\langle g^* \rangle$ of order $|g|$ on $V/\!/g$, such that $g^*$ acts on the subspace $V(g^j)^{\bZ/|g|\bZ}$ by the scalar $e^{2\pi i j/|g|}$, and $(V/\!/g)/\!/g^* \cong V$.
\end{thm}


\begin{lem} \label{lem:orbifold-to-leech}
The three cyclic orbifolds $V_{N(A_1^{24})}/\!/\rho(i)$, $V_{N(A_1^{24})}/\!/\rho(j)$, and $V_{N(A_1^{24})}/\!/\rho(k)$ exist and are isomorphic to $V_\Lambda$.
\end{lem}
\begin{proof}
By Lemma \ref{lem:image-of-binary-octahedral}, the automorphisms $\rho(i), \rho(j), \rho(k)$ are conjugate in $\Aut(V_{N(A_1^{24})})$, so it suffices to prove the claim for $V_{N(A_1^{24})}/\!/\rho(i)$.  By Lemma \ref{lem:fixed-point-for-Neimeier}, the fixed point vertex subalgebra is $V_{\Lambda_0}$, so by Theorem \ref{thm:classification-of-lattice-va-modules} (from \cite{D93}) irreducible modules are the sums $M_\gamma = \bigoplus_{u \in \Lambda_0 + \gamma} \pi_u$ for $\gamma \in \Lambda_0^\vee/\Lambda_0$.  Thus, the strongly holomorphic extensions of $V_{\Lambda_0}$ are precisely the direct sums $V_{\Lambda_0} \oplus M_\gamma$ for $\gamma \neq 0$ such that $M_\gamma$ has integer weight, yielding a lattice vertex algebra.  One choice of $\gamma$ gives non-integer weight, one yields $V_{N(A_1^{24})}$, and the third yields $V_\Lambda$.
\end{proof}


\begin{lem}
The identification of $V_\Lambda$ with $V_{N(A_1^{24})}/\!/\rho(i)$ induces a natural action of a 4-group, generated by a lift of the involution $\rho(j)$ on ${V_{N(A_1^{24})}}^{\rho(i)}$, and an involution $\rho(i)^*$ whose fixed point subalgebra is ${V_{N(A_1^{24})}}^{\rho(i)}$.
\end{lem}
\begin{proof}
The fact that $\rho(i)^*$ is order 2 with specified fixed points is given in Theorem \ref{thm:vEMS}.  The involutions $\rho(j)$ and $\rho(k)$ coincide on this subspace, and commute with the restriction of $\rho(i)^*$, since it is the identity.  There are 2 possible lifts of $\rho(j)$ (seen as an order 2 lift of $-1$ on $\Lambda_0$ to $V_{\Lambda_0}$) to automorphisms of $V_\Lambda$, uniquely determined by where they take a lowest weight vector $v_\lambda \in \pi_\lambda$ for $\lambda \in \Lambda \setminus \Lambda_0$.  The image is necessarily a scalar times $v_{-\lambda}$, and the two choices differ by a sign, i.e., composition with $\rho(i)^*$.  Both lifts have order 2, so we obtain a 4-group.
\end{proof}

The following is a special case of the main result of \cite{S03}.

\begin{lem} \label{lem:automorphism-group-of-fixed-points}
Let $\theta$ be an order 2 lift of the $-1$ automorphism of $\Lambda$ to $V_\Lambda$.  Then, $C_{\Aut V_\Lambda}(\theta)$ has the form $2^{24}.Co_0$, and the automorphism group of $V_\Lambda^\theta$ has the form $2^{24}.Co_1$.
\end{lem}
\begin{proof}
Any automorphism of a vertex operator algebra induces a permutation of the irreducible modules, and in this case, $V_\Lambda^\theta$ has 4 irreducible modules, with pairwise distinct characters.  The pairwise-distinct property implies for any automorphism $\sigma$ of $V_\Lambda^\theta$, there exists a lift to an automorphism of $V_\Lambda$.  We conclude that the automorphism group of $V_\Lambda^\theta$ is equal to the quotient $C_{\Aut V_\Lambda}(\theta)/\langle \theta \rangle$.  Because $\theta$ acts as $-1$ on the tangent space of the torus, the centralizer has torus part $(\pm 1)^{24}$, and because $\theta$ is a lift of the central element $-1$ in $Co_0$, the quotient of the centralizer by the torus part is $Co_0$.  Thus, $\Aut V_\Lambda^\theta$ has the form $2^{24}.Co_0/\langle \theta \rangle$, giving the form $2^{24}.Co_1$.
\end{proof}

\section{The monster vertex algebra}

\begin{defn}
The \textbf{monster vertex algebra} $V^\natural$ is the cyclic order 2 orbifold:
\[V^\natural = V_\Lambda/\!/ \theta,\]
where $\theta$ is an order 2 lift of the $-1$ automorphism of $\Lambda$.  That is, $V^\natural$ decomposes as the direct sum of the simple vertex operator algebra $V_\Lambda^\theta$ and its irreducible module $V_\Lambda(\theta)^\theta$.
\end{defn}


We note some basic facts.

\begin{lem}
\begin{enumerate}
\item $V^\natural$ is a holomorphic vertex operator algebra.
\item $V^\natural$ has character $J(\tau) = q^{-1} + 196884q + \cdots$.  In particular, the weight 1 subspace is zero.
\item $V^\natural$ has an order 2 automorphism $z$ whose fixed points are identified with $V_\Lambda^\theta$.
\end{enumerate}
\end{lem}
\begin{proof}
For the claim about the character, it suffices to work out the characters of the subspaces $V_\Lambda^\theta$ and $V_\Lambda(\theta)^\theta$.  These can be computed from Corollary \ref{cor:characters-for-lattice-involutions}.  The order 2 automorphism $z$ is precisely $\theta^*$ as defined in Theorem \ref{thm:vEMS}.
\end{proof}

\subsection{Triality}

From Lemma \ref{lem:orbifold-to-leech}, we see that $V^\natural$ can be constructed from $V_{N(A_1^{24})}$ as a composite of 2 orbifold duals in 3 different ways: We start by taking the orbifold dual with respect to one of the three involutions $\rho(i)$, $\rho(j)$, $\rho(k)$ to get $V_\Lambda$, then we take the orbifold dual with respect to a lift of one of the other involutions.  This information can be organized in the following way:

\begin{thm} \label{thm:triality}
We have the following diagram of orbifold duals:

\[ \xymatrix{ & &  V_{N(A_1^{24})} \ar@/^/[ddll]^{\rho(i)} \ar@/^/[dd]^{\rho(j)} \ar@/^/[ddrr]^{\rho(k)} \\ \\
V_{N(A_1^{24})}/\!/\rho(i) \ar@/^/[uurr]^{\rho(i)^*} \ar@/^/[ddrr]^{z^*} & & V_{N(A_1^{24})}/\!/\rho(j) \ar@/^/[uu]^{\rho(j)^*} \ar@/^/[dd]^{x^*} & & V_{N(A_1^{24})}/\!/\rho(k) \ar@/^/[uull]^{\rho(k)^*} \ar@/^/[ddll]^{y^*} \\ \\
 & & V_{N(A_1^{24})}/\!/\rho(Q_8) \ar@/^/[uurr]^{y} \ar@/^/[uu]^{x} \ar@/^/[uull]^{z} } \] 

where the vertex operator algebras in the middle row are isomorphic to $V_\Lambda$, the vertex operator algebra on the bottom is isomorphic to $V^\natural$, and $x,y,z$ are commuting order 2 automorphisms of $V^\natural$ that generate a 4-group, which splits $V^\natural$ into a direct sum of 4 irreducible ${V_{N(A_1^{24})}}^{\rho(Q_8)}$-modules $V^{00} = {V_{N(A_1^{24})}}^{\rho(Q_8)}, V^{01}, V^{10}, V^{11}$, satisfying the following properties:
\begin{enumerate}
\item $V^{01}$ is an irreducible ${V_{N(A_1^{24})}}^{\rho(Q_8)}$-submodule of $V_{N(A_1^{24})}(\rho(i))$, on which $z$ acts by $1$ and $x,y$ act by $-1$.  It is contained in $V_{N(A_1^{24})}/\!/\rho(i)$.
\item $V^{10}$ is an irreducible ${V_{N(A_1^{24})}}^{\rho(Q_8)}$-submodule of $V_{N(A_1^{24})}(\rho(j))$, on which $x$ acts by $1$ and $y,z$ act by $-1$.  It is contained in $V_{N(A_1^{24})}/\!/\rho(j)$.
\item $V^{11}$ is an irreducible ${V_{N(A_1^{24})}}^{\rho(Q_8)}$-submodule of $V_{N(A_1^{24})}(\rho(k))$, on which $y$ acts by $1$ and $x,z$ act by $-1$.  It is contained in $V_{N(A_1^{24})}/\!/\rho(k)$.
\item The involutions $x^*, y^*,z^*$ are lifts of the $-1$ automorphism on $\Lambda$, for some identification of the corresponding orbifolds of $V_{N(A_1^{24})}$ with $V_\Lambda$.  In particular, the restriction of $z^*$ to $V_{\Lambda_0}$ is given by $\rho(j)$, a lift of the $-1$ automorphism of the lattice.
\end{enumerate}
Furthermore, the action of $S_3 \subset \Aut V_{N(A_1^{24})}$ permuting $\{ \rho(i), \rho(j), \rho(k) \}$ induces a faithful permutation action on $\{V^{01}, V^{10}, V^{11} \}$, and extends to the conjugation action of a subgroup $\Sigma \subset \Aut V^\natural$ that induces a faithful permutation action of $S_3$ on $\{ x,y,z \}$.
\end{thm}

One can show that the group $\Sigma$, given by lifting the $S_3$ acting on the vertex operator subalgebra ${V_{N(A_1^{24})}}^{\rho(Q_8)}$, can be chosen isomorphic to $S_3$, and together with the 4-group generates a copy of $S_4$.  Indeed, in \cite{FLM88} a pair of generating elements of $S_3 \subset \Aut V^\natural$ was explicitly constructed.



\subsection{Extensions of Conway groups}

We now consider the automorphisms of $V^\natural$ that arise directly from $V_\Lambda$.

\begin{thm}
We have the following diagram of finite groups where arrows are double covers:
\[ \xymatrix{ \hat{C} \ar[rr] \ar[dd] \ar[rd] & & C_0 \ar[dd] \\ & C \ar[rd] \\ C_T \ar[rr] & & C_1 } \]
and the groups have the following form:
\begin{enumerate}
\item $C_0$ is the centralizer of $\theta$ in $\Aut V_\Lambda$.  It has the form $2^{24}.Co_0$, where $2^{24}$ naturally identifies with $\frac{1}{2}\Lambda/\Lambda$.  This group acts faithfully on $V_\Lambda$.
\item $C_1$ is the quotient of $C_0$ by the central subgroup $\langle \theta \rangle$, and has the form $2^{24}.Co_1$.  $C_1$ is the full automorphism group of $V_\Lambda^\theta$.
\item $C_T$ is defined using the projective action of $C_1$ on the twisted module $V_\Lambda(\theta)$.  The projective action on the lowest weight space $T$ (of dimension $2^{12}$) induces an injective homomorphism $C_1 \to PGL(T)$, and the pullback $C_* \subset GL(T)$ has commutator subgroup $C_T$, which has the form $2.2^{24}.Co_1$.  This group acts faithfully on the vertex superalgebra $V_\Lambda^\theta \oplus V_\Lambda(\theta)^{\theta = -1}$.
\item $\hat{C}$ is the pullback $C_0 \times_{C_1} C_T$.  It has the form $2^2.2^{24}.Co_1$, and acts faithfully on the ``generalized vertex algebra'' $V_\Lambda \oplus V_\Lambda(\theta)$ constructed in \cite{H94}.  $C_0$ and $C_T$ are the quotients of $\hat{C}$ by two of the involutions in the central 4-group $2^2$.
\item $C$ is the quotient of $\hat{C}$ by the remaining involution in $2^2$.  $C$ has the form $2^{1+24}.Co_1$, i.e., there is an extraspecial normal subgroup $Q$, naturally identified with $\hat{\Lambda}/2\Lambda$, and the quotient by this subgroup is $Co_1$.  $C$ acts faithfully on $V^\natural$, and it is the centralizer of the distinguished involution $z \in \Aut V^\natural$ that is trivial on $V_\Lambda^\theta$, and $-1$ on $V_\Lambda(\theta)^\theta$.
\end{enumerate}
\end{thm}
\begin{proof}
The identification of $C_0$ and $C_1$ follow from Lemma \ref{lem:automorphism-group-of-fixed-points}.

The ``generalized vertex algebra'' $V_\Lambda \oplus V_\Lambda(\theta)$ is the direct sum of the 4 irreducible $V_\Lambda^\theta$-modules.  Each irreducible module is a simple current, meaning fusion with each irreducible module yields a permutation on the irreducible modules, and the nontrivial simple currents have order 2.  Thus, automorphisms of $V_\Lambda^\theta$ lift to an action of a $2^2$-central extension, where the center acts by scalars $\pm 1$ on the irreducible $V_\Lambda^\theta$-modules in a manner compatible with fusion.  We call this central extension $\hat{C}$, and we have the diagram of double covers coming from adjoining simple currents.

We have already identified $C_0$ as the quotient that naturally acts on $V_\Lambda$, so among the groups $C_T$ and $C$, one of them naturally acts faithfully on $V_\Lambda^\theta \oplus V_\Lambda(\theta)^{\theta = -1}$ and one of then naturally acts faithfully on $V^\natural$.  To distinguish them, we note that $T$ is the weight $3/2$ space in $V_\Lambda(\theta)$, so the linear action of $C_T$ on $T$ extends to an action on $V_\Lambda(\theta)^{\theta = -1}$.

We now have a subgroup $C$ of $\Aut V^\natural$, with a distinguished central element $z$ that acts trivially on $V_\Lambda^\theta$ and by $-1$ on $V_\Lambda(\theta)^\theta$.  The centralizer of $z$ in $\Aut V^\natural$ is then given by those automorphisms that preserve the decomposition into the direct sum of $V_\Lambda^\theta$ and $V_\Lambda(\theta)^\theta$.  Since the action of $C$ preserves this decomposition, $C \subseteq C_{\Aut V^\natural}(z)$.  Conversely, restriction to $V_\Lambda^\theta$ induces a homomorphism from $C_{\Aut V^\natural}(z)$ to $C_1$ with kernel generated by $z$, and the restriction to $C$ is surjective with kernel containing $z$.  Thus, $C = C_{\Aut V^\natural}(z)$.
\end{proof}

We briefly study the structure of the finite group $C$.

\begin{lem} \label{lem:normal-subgroups-of-C}
The only normal subgroups of $C$ are $\{1\}$, $\langle z \rangle$, $Q$, and $C$.
\end{lem}
\begin{proof}
This follows from the simplicity of $Co_1$ and the nonexistence of a $Co_1$-stable proper subgroup of $Q$ strictly containing $\langle z \rangle$.
\end{proof}

\begin{lem} \label{lem:commutes-with-x}
An element $t \in Q$ commutes with $x$ if and only if its image in $Q/\langle z \rangle = \frac{1}{2}\Lambda/\Lambda$ lies in the image of $\frac{1}{2}\Lambda_0 \subset \frac{1}{2}\Lambda$.
\end{lem}
\begin{proof}
We first note that the property of commuting with $x$ does not depend on multiplication by $z$, since $z$ commutes with $x$.  We therefore consider the image of $t$ in $Q/\langle z \rangle$, identified with $\frac{1}{2}\Lambda/\Lambda$, and let $h \in \frac{1}{2}\Lambda$ denote a lift.

Restricting $x$ to $(V^\natural)^z = V_\Lambda^\theta$ yields the automorphism that is identity on $(\pi_u + \pi_{-u})^\theta$ for $u \in \Lambda_0$ and $-1$ on the remaining modules, so this restriction is equal to the automorphism given by restriction of $e^{2\pi i (\frac{1}{4}\alpha_{\cC})_0}$ on $V_\Lambda$.  We conclude that $x$ lies in $Q$, since the image of $x$ under the double cover $C \to 2^{24}.Co_1$ lies in the normal subgroup $2^{24}$.  Furthermore, setting $s = \frac{1}{2}\alpha_{\cC}$, we see that the image of $x$ has the form $e^{2\pi i s_0}$, where $(2s, 2h) \in 2\bZ$ if and only if $h \in \frac{1}{2}\Lambda_0$.  By the commutator rule for double covers of lattices, we see that $e^{2\pi i s_0}$ and $e^{2\pi i h_0}$ lift to commuting elements $x$ of $Q$ if and only if $h \in \frac{1}{2}\Lambda_0$.
\end{proof}

\subsection{Finiteness}

Our proof that $\Aut V^\natural$ is a finite group is surprisingly easy - compare with the proof in \cite{T84} or \cite{G82}.

\begin{thm}
The group $\Aut V^\natural$ of all vertex operator algebra automorphisms of $V^\natural$ is finite.
\end{thm}
\begin{proof}
By Theorem 3.8 in \cite{L98}, all strongly rational vertex operator algebras are finitely generated, so $V^\natural$ is finitely generated.  By taking the sum of weight subspaces containing a generating set, we see that $V^\natural$ is generated by a finite dimensional subspace that is invariant under all automorphisms.

The condition that a linear transformation on a generating subspace induces an automorphism of the vertex operator algebra is given by polynomial identities on the coefficients of multiplication, so $\Aut V^\natural$ embeds as a Zariski closed subgroup of $GL_n(\bC)$ for some $n$.

We now show that the Lie algebra of this group is trivial.  Recall that we have distinguished involutions $x,y,z$ that are mutually conjugate, and their centralizers are finite (in particular, isomorphic to $C$).  Thus, conjugation by any of these elements acts as $-1$ on the Lie algebra of $\Aut V^\natural$, hence as inversion on the connected component of identity.  However, the equation $xy = z$ implies $z$ also acts as identity on the Lie algebra, so the Lie algebra of $\Aut V^\natural$ is trivial.

We conclude that $\Aut V^\natural$ is a zero dimensional linear algebraic group, hence finite.
\end{proof}

\subsection{Simplicity}

We now show that $\Aut V^\natural$ is simple.  Since we have a large subgroup $C = C_{\Aut V^\natural}(z)$ which has relatively few normal subgroups, we will consider the possible intersections between $C$ and a normal subgroup of $\Aut V^\natural$.

\begin{lem} \label{lem:case-intersection-is-trivial}
Let $H$ be a normal subgroup of $\Aut V^\natural$.  If $H \cap C$ is trivial, then $H$ is trivial.
\end{lem}
\begin{proof}
If $H \cap C = \{1\}$, then conjugation by $z$ is an involution on $H$ that has no non-identity fixed points.  By a standard result (using injectiveness hence bijectiveness of the map $a \mapsto zaza^{-1}$ on $H$), this implies $H$ is abelian of odd order and $z$ acts by inversion.  The same is true of $x$ and $y$, since they are conjugate to $z$ and $H$ is normal.  However, $xy = z$, so $z$ also acts by identity, and $H$ is necessarily trivial.
\end{proof}


\begin{lem} \label{lem:case-intersection-is-cyclic}
If $H$ is a normal subgroup of $\Aut V^\natural$, then $H \cap C \neq \langle z \rangle$.
\end{lem}
\begin{proof}
$z$ is conjugate to $x$ in $\Aut V^\natural$, and $x \in C$, so $x \not\in \langle z \rangle$ contradicts normality of $H$.
\end{proof}


\begin{lem} \label{lem:case-intersection-is-Q}
If $H$ is a normal subgroup of $\Aut V^\natural$, then $H \cap C \neq Q$.
\end{lem}
\begin{proof}
It suffices to show that there is an element of $\Aut V^\natural$ that conjugates an element of $Q$ to an element of $C \setminus Q$.  For this, we consider the involution $\rho(\frac{i+j}{\sqrt{2}}) \in \rho(2O)$ by which conjugation switches $\rho(i)$ with $\rho(j)$.  By Theorem \ref{thm:triality}, there exists a lift to an automorphism $\sigma$ of $V^\natural$ such that conjugation by $\sigma$ switches $z$ with $x$.  We will consider a suitable $\tilde{u} \in Q$ that commutes with $x$, because conjugation by $\sigma$ will then yield an element that commutes with $z$ and hence lies in $C$.

Let $u = ( u^1,\ldots,u^{24}) \in SL_2(\bC)^{24}$ be an automorphism of $V_{N(A_1^{24})}$ of the form $e^{2\pi i h_0}$ for $h \in \frac{1}{2}\Lambda_0 \setminus (\Lambda \cup N(A_1^{24})) \subset \pi_0$.  Then, by the relation $u^a = \begin{pmatrix} i & 0 \\ 0 & -i \end{pmatrix}^{\sqrt{8}h^a}$, each $u^a$ lies in the exponent 4 subgroup $\langle \begin{pmatrix} i & 0 \\ 0 & -i \end{pmatrix} \rangle$, and the set $c = \{a \in \{1,\ldots,24\} | u^a \not\in \pm I_2\}$ is a Golay codeword of length 8, 12, or 16.

Because $h \in \frac{1}{2}N(A_1^{24}) \setminus N(A_1^{24})$, $u$ has order 2.  The automorphism $e^{2\pi i h_0}$ of $V_{N(A_1^{24})}/\!/\rho(i) \cong V_\Lambda$ is a lift of $u$, and because $h$ lies in $\frac{1}{2}\Lambda \setminus \Lambda$, this lift also has order 2, and in particular, lies in the normal subgroup $2^{24} \subset 2^{24}.Co_0$.  We choose some lift $\tilde{u} \in Q \setminus \langle z \rangle$ of $e^{2\pi i h_0}$.  By Lemma \ref{lem:commutes-with-x}, $\tilde{u}$ commutes with $x$.  Thus, $\sigma \tilde{u} \sigma^{-1}$ commutes with $\sigma x \sigma^{-1} = z$, so $\sigma \tilde{u} \sigma^{-1} \in C$.

It remains to show that $\sigma \tilde{u} \sigma^{-1} \not\in Q$.  Conjugating $e^{2\pi i h_0}$ by $\rho(\frac{i+j}{\sqrt{2}})$ yields an element $w = (w^1,\ldots,w^{24}) \in SL_2(\bC)^{24}$ with entries given by $w^a = \begin{cases} u^a & a \not\in c \\ \pm \begin{pmatrix} 0 & i \\ i & 0 \end{pmatrix} & a \in c \end{cases}$.  We see that $w$ acts on $\bR^{24} \subset V_\Lambda$ by a composite of root reflections parametrized by the Golay codeword $c$ whose weight is neither 0 nor 24.  In particular, $w$ acts nontrivially on $\pi_0^\theta$, so $w$ does not lift to an element of $2^{24} \subset 2^{24}.Co_0$ or to $Q$.  On the other hand, the lift $\sigma \tilde{u} \sigma^{-1}$ of $w$ commutes with $z$, because $\tilde{u}$ commutes with $x$, so $\sigma \tilde{u} \sigma^{-1}$ lies in $C \setminus Q$.
\end{proof}

\begin{lem} \label{lem:case-intersection-is-C}
The only normal subgroup of $\Aut V^\natural$ containing $C$ is $\Aut V^\natural$.
\end{lem}
\begin{proof} (from \cite{T84})
Let $H$ be a normal subgroup of $\Aut V^\natural$ containing $C$, let $S$ be a Sylow 2-subgroup of $C$, and let $N$ be the normalizer of $S$ in $H$.  Since the center of $S$ is $\langle z \rangle$, $N$ centralizes $z$, hence lies in $C$.

We claim that $S$ is a Sylow 2-subgroup of $H$.  Indeed, otherwise there would be a 2-group in $H$ strictly containing $S$ but contained in $N$, contradicting Sylow's theorem.  From this, the Frattini argument yields $H= \Aut V^\natural$.  To elaborate, because $H$ is normal, and all Sylow 2-subgroups of $H$ are conjugate to each other, for any $g \in \Aut V^\natural$, there is some $h \in H$ such that $S^g = S^h$, implying $g \in hN \subseteq H$. 
\end{proof}

\begin{thm}
$\Aut V^\natural$ is simple.
\end{thm}
\begin{proof}
If $H$ is a normal subgroup of $\Aut V^\natural$, then the intersection $H \cap C$ is normal in $C$, so by Lemma \ref{lem:normal-subgroups-of-C} it must be one of $\{1\}$, $\langle z \rangle$, $Q$, and $C$.  By Lemmata \ref{lem:case-intersection-is-trivial}, \ref{lem:case-intersection-is-cyclic}, \ref{lem:case-intersection-is-Q}, \ref{lem:case-intersection-is-C}, we find that $H$ is either trivial or equal to $\Aut V^\natural$.
\end{proof}

\section{What else can we derive about the monster from vertex-algebraic methods?}

We've proved our main objective, that $\Aut V^\natural$ is finite and simple, using modern methods in vertex algebras and mostly 19th century methods in finite groups.  However, there is much more one might want to say.  For example, one may want to compute the order of this group, and figure out some of its irreducible characters.  For the monster, this information was computed even before the group was constructed, using group-theoretic methods that are well beyond the scope of this paper.

We will describe some facts that we can prove ``easily modulo vertex algebras and some hard manipulations with complex functions''.

\begin{thm} \label{thm:monstrous-moonshine} (Monstrous moonshine minus explicit computations)
For any $g \in \Aut V^\natural$, the McKay-Thompson series
\[ T_g(\tau) = \sum_{n \geq 0} \Tr(g|V^\natural_n) q^{n-1} \]
is a completely replicable function of order $|g|$, whose $k$-th replicate is $T_{g^k}(\tau)$.  That is, for any $k \geq 1$, the function $\sum_{ad = k, 0 \leq b < d} T_{g^a}(\frac{a\tau+b}{d})$ is equal to the unique polynomial of degree $k$ in $T_g(\tau)$ that has $q$-expansion of the form $q^{-k} + O(q)$.  Furthermore, each $T_g$ is a hauptmodul, i.e., there is a subgroup $\Gamma_g \subset SL_2(\bR)$ containing $\Gamma_0(N)$ for some $N$, such that the quotient $\fH/\Gamma_g$ of the upper half-plane is a genus zero Riemann surface with finitely many punctures, and $T_g$ generates its function field.
\end{thm}
\begin{proof}
The first claim is proved in \cite{B92} up to the end of section 8, before the explicit identification of McKay-Thompson series with the candidate modular functions presented in \cite{CN79}.  The second claim follows from the combination of two theorems: the main theorem of \cite{M94} asserts that completely replicable series that are ``$J$-final'' are modular functions invariant under $\Gamma_0(N)$ for some $N$ whose prime divisors are the same as those of $|g|$, and the main theorem of \cite{CG97} asserts that completely replicable functions are either Laurent polynomials in $q$ or hauptmoduls.
\end{proof}

We note that the completely replicable functions of finite order were completely classified in \cite{ACMS92}: there are 171 monstrous functions and 157 non-monstrous functions.

\begin{cor}
The primes dividing the order of $\Aut V^\natural$ are a subset of the 15 ``supersingular primes'' $2, 3, 5, 7, 11, 13, 17, 19, 23, 29, 31, 41, 47, 59, 71$.
\end{cor}
\begin{proof}
The completely replicable functions of prime order $p$ are hauptmoduls of level a power of $p$, and this implies the modular curve $X_0(p)^+$ is genus zero.  The 15 primes satisfying this genus zero condition were classified in \cite{O74}.
\end{proof}

We can also constrain the ``largeness'' of  $\Aut V^\natural$ by analyzing the McKay-Thompson series.  In \cite{T93}, Tuite proposed an orbifold duality correspondence between fixed-point free classes in $Co_1$ satisfying a ``no massless states'' condition and non-Fricke classes in the monster.  This was proved in \cite{C17}, but the proof only requires eta-product expansions of non-Fricke Hauptmoduls and Frame shapes in $Co_1$, so no explicit knowledge of the monster is necessary.  Using this correspondence, we can narrow down the possible McKay-Thompson series to a set of 174 functions (where the true set is 171) without explicit use of the monster or structure of $V^\natural$:

\begin{thm} \label{thm:174} (\cite{CKU17})
Let $V$ be a strongly holomorphic vertex operator algebra with graded dimension $J$, and let $g$ be a finite order automorphism of $V$.  Then the graded trace of $g$ on $V$ is one of 174 completely replicable functions.
\end{thm}

For the primes $p$ in $\{2,3,5,7,13\}$, we obtain the precise $p$-valuation of the order of $\Aut V^\natural$ using cyclic orbifolds\footnote{We are thankful to H. Shimakura for explaining this argument.}:

\begin{thm} \label{thm:valuation-for-small-primes}
If $p = 2$ (resp. $3,5,7,13$), the Sylow $p$-subgroup of $\Aut V^\natural$ has order $2^{46}$ (resp. $3^{20}, 5^9, 7^6, 13^3$).
\end{thm}
\begin{proof}
For each $p \in \{2,3,5,7,13\}$, there is an order $p$ automorphism $\sigma_p$ of $V_\Lambda$ lifting a fixed-point free automorphism of $\Lambda$, that yields a construction of $V^\natural$ (by \cite{CLS16} for $p=3$ and \cite{ALY17} for $p=5,7,13$), and we obtain a faithful action of a subgroup $G_p$ of $V^\natural$ of the form $p^{1+\frac{24}{p-1}}.(C_{Co_0}(\sigma_p)/\langle \sigma_p \rangle)$.  We find that $G_p$ is the centralizer of the order p element $\sigma_p^*$, and $G_p$ has $p$-Sylow subgroup $S_p$ with center $Z_p$ of order $p$.  However, any Sylow $p$-subgroup of $\Aut V^\natural$ containing $S_p$ has nontrivial center contained in $Z_p$, so the centralizer of $Z_p$ in $\Aut V^\natural$ contains a $p$-Sylow subgroup of $\Aut V^\natural$.  An analysis of centralizers in $Co_1$ yields the orders we want.
\end{proof}

Finally, we obtain the precise $p$-valuation of the order of $\Aut V^\natural$ for large primes, by applying Norton's analysis of consequences of generalized moonshine in \cite{N01}:

\begin{thm}
If $p \geq 17$, then any nontrivial Sylow $p$-subgroup of $\Aut V^\natural$ is cyclic of order $p$.
\end{thm}
\begin{proof}
From Theorem \ref{thm:174}, there are no elements of order $p^2$, and if $g \in \Aut V^\natural$ has order $p$, then the McKay-Thompson series $T_g(\tau)$ is invariant under the Fricke involution $\tau \mapsto -1/\tau$.  Thus, the generalized moonshine functions $Z(g,h;\tau) = \sum \Tr(h|V^\natural(g)_n)q^{n-1}$, for $g$ of order $p$ and $h$ commuting with $g$, have expansions of the form $a q^{-1/p} + O(q^{1/p})$ for $a$ a root of unity, and are Hauptmoduln of some finite level $N$ \cite{C12}.  We shall assume that $g$ and $h$ generate an elementary abelian subgroup of order $p^2$, and derive a falsehood.

Because the McKay-Thompson series of all order $p$ automorphisms of $V^\natural$ have poles at zero, $Z(g,h;\tau)$ is singular at all cusps.  Let $G = \{ \left( \begin{smallmatrix} a & b \\ c & d \end{smallmatrix} \right) \in SL_2(\bR) | Z(g,h;\tau) = Z(g,h;\frac{a\tau+b}{c\tau+d}) \}$ denote the fixing group of $Z(g,h;\tau)$.  By the Hauptmodul property, the modular curve $X_{g,h} = \fH/G$ has a single cusp, and from the $q$-expansion, the single resulting cusp has width $p$.  By Theorem I of \cite{DLN12}, we have $N = p$, i.e., $G$ contains $\Gamma(p)$.   Thus, we have a genus zero group $G$ containing $\Gamma(p)$ such that the corresponding quotient map $X(p) \to X_{g,h}$ is unramified at all cusps.  By applying the orbit-stabilizer theorem to the action on cusps of $X(p)$, we find that $[G:\Gamma(p)] = p^2-1$, and $G$ is contained in $\Gamma(1)$ as an index $p$ subgroup.  Then, the image of $G$ in $\Gamma(1)/\Gamma(p)$ is an index $p$ subgroup of $SL_2(\bF_p)$, which (by a theorem of Galois) only exists when $p \leq 11$.  This contradicts our assumption on $p$.



\end{proof}

We therefore have sharp upper bounds on the valuation of $|\Aut V^\natural|$ for every prime except 11.  In the other direction, we can produce lower bounds on the size of $\Aut V^\natural$.  This is important, because in the paper \cite{O74} where he classified the supersingular primes, Ogg asked why the set of primes dividing the order of the monster coincided with the set of primes $p$ for which $X_0(p)^+$ is genus zero, offering a bottle of Jack Daniels for a satisfactory answer.  Borcherds's proof of the monstrous moonshine conjecture for $V^\natural$ gives half of the answer, and the other half requires an explanation for why the monster is so large.  We can produce, without appealing to much group theory, a lower bound on the order that is a factor of $11 \cdot 17 \cdot 19$ off the true size:

\begin{thm} \label{thm:order}
The order of $\Aut V^\natural$ is divisible by $2^{46}\cdot 3^{20} \cdot 5^9 \cdot 7^6 \cdot 11 \cdot 13^3 \cdot 23 \cdot 29 \cdot 31 \cdot 41 \cdot 47 \cdot 59 \cdot 71$.
\end{thm}
\begin{proof}
For the primes $2,3,5,7,13$, Theorem \ref{thm:valuation-for-small-primes} gives the precise valuation.  Since the subgroup $C$ has $Co_1$ as a quotient, we also have divisors $11$ and $23$.  Now, we consider the large primes:
\begin{enumerate}
\item The weight 2 space $V^\natural_2$ splits into the direct sum of a trivial representation spanned by the conformal vector $\omega$ and the 196883-dimensional space of Virasoro primary vectors, namely those annihilated by all $L_n$, $n>0$.  This primary space decomposes into irreducible $C$-modules as $299$ from $\Sym^2 \fh$, $98280$ from norm 4 vectors in $\Lambda$, and $24\cdot 2^{12}$ from the twisted module, and it is not hard to show using triality that it forms an irreducible representation of $\Aut V^\natural$.  The dimension of an irreducible representation of any finite group is a divisor of its order, so we find that the prime factors of $196883 = 47 \cdot 59 \cdot 71$ are therefore divisors of the order of $\Aut V^\natural$.
\item The primary subspaces of weight 3 and 4 are also irreducible, yielding additional prime factors $29, 31, 41$ from their dimensions.
\end{enumerate}
\end{proof}

We will consider the remaining factors of $11$, $17$ and $19$ in the ``open problems'' section.

\section{Open problems}

\begin{enumerate}
\item Are there alternative, non-triality vertex-algebraic methods to proving the finiteness and simplicity of $\Aut V^\natural$?  I have tried applying distinct cyclic orbifolds, e.g., of orders 2 and 3, and also replacing the 4-group with an elementary $p$-group for $p \in \{3,5,7,13\}$.  Finiteness seems to work out with the elementary groups, but I was unable to produce satisfactory arguments for simplicity.
\item Are there analogous vertex-algebraic methods to prove simplicity of other members of the ``happy family'' of sporadic groups that are subquotients of $\bM$?  For example, the Baby monster $\bB$ has an order 2 element with centralizer $2^{1+22}.Co_2$, where $Co_2$ is the stabilizer of a norm 4 vector in the automorphism group of $\Lambda$, suggesting the existence of a parallel construction.  Furthermore, there is a vertex operator algebra with $\bB$ symmetry, constructed in \cite{H95} as the commutant of an ``Ising model'' Virasoro vertex subalgebra of $V^\natural$ with central charge $1/2$.  There are other explicitly constructible vertex operator subalgebras with sporadic symmetry, e.g., Thompson symmetry from the commutant of a 3C pair of Ising vectors, and there are also twisted modules with natural sporadic symmetry, so there are multiple paths to proving simplicity and other properties.
\item Can we use the Conway moonshine vertex superalgebras $V^{f\natural}$ from \cite{D05} and $V^{s\natural}$ from \cite{DM14} to prove the simplicity of $Co_1$ and its subquotients?  Both $V^{f\natural}$ and $V^{s\natural}$ can be constructed as the vertex superalgebra attached to the odd unimodular lattice $D_{12}^+$, and have automorphism group $\Spin(24)$ as vertex superalgebras, but the stabilizers of the ``canonical'' $N=1$ superconformal vector are $Co_1$ and $Co_0$, respectively.
\item Can we extract the remaining factors $11,17,19$ in the order of the monster (i.e., those not seen in Theorem \ref{thm:order}) by vertex algebraic methods?  There are a few possible methods:
\begin{enumerate}
\item We can try dimensions of irreducible representations in $V^\natural$.  The main problem is that higher weight primary subspaces become ``very reducible'', and this doesn't seem to be much easier than computing the whole character table.  As an extreme example, the $16$th smallest irreducible representation has dimension $5^9 \cdot 7^6 \cdot 11^2 \cdot 17 \cdot 19$, and would solve all of our problems, but a short computation with the character table of the monster and complete replicability shows that this representation first appears in $V^\natural_{60}$, with multiplicity $4$ \cite{GAP}.  There are many other irreducible representations with useful divisibility properties that appear at lower weight, but none of them are particularly convenient.  We may be able to extract distinguished subspaces of primary spaces using the monster Lie algebra, but this idea is still undeveloped.
\item On the other hand, we may be able to find these factors in twisted modules.  For example, if $g$ is Fricke-invariant of order 7, then the weight $8/7$ subspace of $V^\natural(g)$ has dimension $3 \cdot 17$, and we would like to show that the centralizer of $g$ acts irreducibly.  If we could identify the centralizer of $g$ with $7 \times He$, then we could use the fact that the smallest faithful representation of the Held group has dimension 51 \cite{ATLAS}.
\item We can use triality.  For example, the centralizer of an order 11 element is $11 \times M_{12}$, and a vertex-algebraic way to detect this would give us a way to see the factor of $11^2$ in the order.  Here is what an attempt to generate $M_{12}$ using triality might look like: there is an order 11 automorphism of $V_{N(A_1^{24})}$ that commutes with $\rho(Q_8)$, and the centralizer of an element of order 11 in $C$ has the form $2^{1+4}_+.S_3 \times 11$ \cite{ATLAS}.  We therefore have 3 copies of the order $2^6\cdot 3$ group $2^{1+4}_+.S_3$ whose intersection is $2^2$, acting on the fixed-points of an order 11 element in $V^\natural$, and we need to show that they generate $M_{12}$ (which has order $2^6\cdot 3^3 \cdot 5 \cdot 11$).
\end{enumerate}
\item How much of the character table of the monster can be extracted from the information in section 5?  We have a set of 174 potential McKay-Thompson series as an upper bound on ``conjugacy classes modulo some equivalence relation'', but not much information about how large the equivalence classes can be.  We also have concrete inputs from transporting a large number of classes from the subgroup $C$, and we have pure existence of more classes attached to large primes.  We can also derive finer-grained information about tensor products and exterior powers from the multiplication structure of $V^\natural$ and the monster Lie algebra.  If we ignore the fact that we don't have the precise order, this appears to be much more information than Fischer, Livingstone, and Thorne had when they computed the full character table in 1978 \cite{T78}.  Thus, the answer to our literal question is likely to be ``all of it, given enough computational effort'', but one can also reasonably ask for something like a roadmap.
\end{enumerate}

\end{document}